\documentclass[12pt]{amsart}
\usepackage{amssymb, amsthm, amsmath}

\newtheorem{question}{Question}[section]
\newtheorem{theorem}[question]{Theorem}
\newtheorem{lemma}[question]{Lemma}

\newtheorem{definition}[question]{Definition}

\newtheorem{proposition}[question]{Proposition}
\newtheorem{example}[question]{Example}

\title{Countably compact weakly Whyburn spaces}
\author{Santi Spadaro}
\address{Instituto de Matematica e Estatistica (IME-USP) \\ Universidade de Sao Paulo \\ Rua do Matao, 1010 - Cidade Universitaria \\ 05508-090 Sao Paulo - SP \\ Brazil}
\email{santidspadaro@gmail.com}

\subjclass[2000]{Primary: 54A20, 54A25; Secondary: 54G10}

\keywords{weak Whyburn property, convergence, Lindel\"of $P$-space, Urysohn, countably compact, pseudoradial}

\begin{document}

\begin{abstract}
The weak Whyburn property is a generalization of the classical sequential property that was studied by many authors. A space $X$ is weakly Whyburn if for every non-closed set $A \subset X$ there is a subset $B \subset A$ such that $\overline{B} \setminus A$ is a singleton. We prove that every countably compact Urysohn space of cardinality smaller than the continuum is weakly Whyburn and show that, consistently, the Urysohn assumption is essential. We also give conditions for a (countably compact) weakly Whyburn space to be pseudoradial and construct a countably compact weakly Whyburn non-pseudoradial regular space, which solves a question asked by Bella in private communication. 
\end{abstract}

\maketitle

\section{Introduction}

The letter $X$ denotes a Hausdorff topological space. The weak Whyburn property is a natural refinement of the sequential property that relies on the following generalization of convergent sequence.

\begin{definition}
Given a space $X$, we say that $A \subset X$ is \emph{an almost closed set converging to $x$} if $\overline{A} \setminus A=\{x\}$.
\end{definition}

\begin{definition}
A space $X$ is called weakly Whyburn if for every non-closed set $A \subset X$ there is an almost closed set $B \subset A$ which converges to a point outside of $A$.
\end{definition}

Weakly Whyburn spaces were introduced by Simon \cite{S1} under the name of WAP spaces.

Clearly, every sequential space is weakly Whyburn. Let $p \in \beta \omega \setminus \omega$. Then the space $\omega \cup \{p\}$ with the topology induced from the \v{C}ech-Stone compactification of the integers $\beta \omega$ shows that the previous implication cannot be reversed. Bella \cite{B} noted that every compact weakly Whyburn space is sequentially compact and satisfies the following weakening of sequentiality.

Recall that a transfinite sequence $\{x_\alpha: \alpha < \kappa \} \subset X$ (where $\kappa$ is a cardinal) is said to converge to a point $x \in X$ if for every open neighbourhood $U$ of $x$ there is an ordinal $\beta$ such that $\{x_\alpha: \alpha > \beta \} \subset U$.

\begin{definition}
A space $X$ is called \emph{pseudoradial} if for every non-closed set $A \subset X$, there is a point $x \in \overline{A} \setminus A$ and a transfinite sequence $\{x_\alpha: \alpha < \kappa\} \subset A$ such that $\{x_\alpha: \alpha < \kappa \}$ converges to $x$.
\end{definition}

Bella \cite{B} proved that every compact weakly Whyburn space is pseudoradial and Dow \cite{D} constructed under $\diamondsuit$ an example of a compact pseudoradial non-weakly Whyburn space. The existence of such a space within the usual axioms of ZFC is still an open problem. Recall that a $P$-space is a topological space where every $G_\delta$ set is open. The behavior of Lindel\"of $P$-spaces is closer to that of compact spaces than that of general Lindel\"of spaces. For example, a countable product of Lindel\"of $P$-spaces is Lindel\"of and every Lindel\"of $P$-space is normal. In \cite{BCS} we proved a version of Bella's result for Lindel\"of $P$-spaces: every Lindel\"of weakly Whyburn $P$-space such that $\psi(X) < \aleph_\omega$ is pseudoradial. It is still an open question whether there exists a Lindel\"of weakly Whyburn non-pseudoradial $P$-space. 

It would be interesting to know whether the Lindel\"of property can be weakened to countable extent (that is, every closed discrete set is countable) in our result. We prove that this is the case for spaces of character at most $\omega_2$. We also construct an example of a countably compact weakly Whyburn regular space which is not pseudoradial, thus showing that compactness cannot be weakened to countable compactness in Bella's result. This answers a question asked by Bella in private communication.

In Section 2 we prove that every countably compact Urysohn  space of cardinality smaller than the continuum is weakly Whyburn and show that the Urysohn assumption is essential by constructing a countably compact Hausdorff non-weakly Whyburn space of cardinality $\omega_1$.

In our proofs we will sometimes use elementary submodels of the structure $(H(\mu), \epsilon)$. Dow's survey \cite{Des} is enough to read our paper, and we give a brief informal refresher here. Recall that $H(\mu)$ is the set of all sets whose transitive closure has cardinality smaller than $\mu$. When $\mu$ is regular, $H(\mu)$ is known to satisfy all axioms of ZFC, except for the power set axiom, but as long as $\mu$ is large enough to contain everything we need in our proof, this will not be a problem. We say, informally, that a formula is satisfied by a set $S$ if it is true when all bounded quantifiers are restricted to $S$. A set $M \subset H(\mu)$ is said to be an elementary submodel of $H(\mu)$ (and we write $M \prec H(\mu)$) whenever a formula with parameters in $M$ is satisfied by $H(\mu)$ if and only if it is satisfied by $M$. 

The downward Lowenheim-Skolem theorem guarantees that for every $S \subset H(\mu)$, there is an elementary submodel $M \prec H(\mu)$ such that $S \subset M$ and $|M| \leq |S| \cdot \omega$. This theorem is sufficient for many applications, but it is often useful for $M$ to satisfy some kind of closure property. For example, $M$ is said to be \emph{$\omega$-covering} if for every $A \in [M]^\omega$ there is $B \in M$ such that $A \subset B$ and $|B| \leq \aleph_0$. Given a large enough regular $\mu$ and $S \in [H(\mu)]^\omega$ there is $M \prec H(\mu)$ such that $S \subset M$, $M$ is $\omega$-covering and $|M|=\aleph_1$.

The following theorem is also used often: let $M$ be an elementary submodel of a large enough $H(\mu)$ such that $\kappa + 1 \subset M$ and let $S$ be a $\kappa$-sized element of $M$. Then $S$ is a subset of $M$ (in particular, every countable element of an infinite elementary submodel is always a subset of it).

All undefined notions can be found in \cite{E} for topology and \cite{Ku} for set theory.

\section{When is a small countably compact space weakly Whyburn?}

Recall that a space is Urysohn if for every pair of points $x \neq y$ there are open neighbourhoods $U$ of $x$ and $V$ of $y$ such that $\overline{U} \cap \overline{V}=\emptyset$.

\begin{theorem} \label{mainwhyburn}
Every countably compact Urysohn space of cardinality $<2^{\aleph_0}$ is weakly Whyburn.
\end{theorem}

\begin{proof}
Let $X$ be a countably compact Urysohn space. Suppose that $X$ is not weakly Whyburn. Then there is a non-closed set $A \subset X$ such that no almost closed set converges outside of it.

Set $A_\emptyset=A$. Suppose that for some $n<\omega$ we have constructed subsets $\{A_f: f \in 2^{\leq n}\}$ of $A$ such that:

\begin{enumerate}
\item $\overline{A_f} \setminus A \neq \emptyset$, for every $f \in 2^{\leq n}$.
\item For every $i \leq n$ and for every $f,g \in 2^i$ such that $f \neq g$, we have $\overline{A_f} \cap \overline{A_g}=\emptyset$.
\item $A_f \subset A_g$, whenever $f \supset g$.
\end{enumerate}

Given $f \in 2^n$, let $x_{f^\frown 0}, x_{f^\frown 1}$ be distinct points such that $x_{f^\frown 0}, x_{f^\frown 1} \in \overline{A_f} \setminus A$. Let $U_{f^\frown 0}$ and $U_{f^\frown 1}$ be neighbourhoods of $x_{f^\frown 0}$ and $x_{f^\frown 1}$ respectively such that $\overline{U_{f^\frown 0}} \cap \overline{U_{f^\frown 1}}=\emptyset$ and let $A_{f^\frown 0}=U_{f^\frown 0} \cap A_f$ and $A_{f^\frown 1}=U_{f^\frown 1} \cap A_f$.

Let $\{A_f: f \in 2^{<\omega}\}$ be the family obtained at the end of the induction. For every $f \in 2^\omega$, use countable compactness to pick a point $x_f \in \bigcap\{\overline{A_{f \upharpoonright n}}: n < \omega \}$.

We claim that the map $\phi: 2^\omega \to X$ defined by $\phi(f) = x_f$ is one-to-one and hence $|X| \geq 2^{\aleph_0}$. Indeed, if $f \neq g$ then there is $n<\omega$ such that $f \upharpoonright n \neq g \upharpoonright n$. Now $x_f \in \overline{A_{f \upharpoonright n}}$ and $x_g \in \overline{A_{g \upharpoonright n}}$ and $\overline{A_{f \upharpoonright n}} \cap \overline{A_{g \upharpoonright n}}=\emptyset$, which implies that $x_f \neq x_g$.
\end{proof}

The above theorem cannot be extended to Hausdorff spaces, unless CH holds. The counterexample is a variation on a space from \cite{BvDMW}. First of all, we need the following lemma. 

Recall that a weak $P$-space is a space where countable sets are closed.

\begin{lemma} \label{lemexample}
There is a regular weak $P$-space $(Z, \rho)$ that is not weakly Whyburn.
\end{lemma}

\begin{proof}
Let $X=\sigma(2^{\omega_1})=\{f \in 2^{\omega_1}: |f^{-1}(1)|<\aleph_0 \}$, with the topology induced from the topology generated by the $G_\delta$ subsets of $2^{\omega_1}$. Let $X'$ be a disjoint copy of $X$ and define a topology $\rho$ on $Z=X \cup X'$ as follows. Fix a bijection $F: X \to X'$. We declare all points of $X'$ to be isolated and a basic neighbourhood of a point $x \in X$ to be of the form $U \cup F(U) \setminus F(N)$, where $U$ is a neighbourhood of $x$ in $X$ and $N$ is a nowhere dense subset of $X$.

\noindent {\bf Claim 1}. $(Z, \rho)$ is a weak $P$-space.

\begin{proof}[Proof of Claim 1] Let $C \subset Z$ be a countable set. Since $X$ is a $P$-space, the set $N:=F^{-1}(C \cap X')$ is nowhere dense.

Let $x \notin C$. If $x \in X'$, then $\{x\}$ is a neighbourhood of $x$ disjoint from $C$, while if $x \in X$ then, since $X$ is a $P$-space, there is an open set $U \subset X$ (with its usual topology) such that $U \cap (C \cap X)=\emptyset$. Therefore $U \cup F(U) \setminus F(N)$ is an open neighbourhood of $x$ disjoint from $C$.
\renewcommand{\qedsymbol}{$\triangle$}
\end{proof}

\noindent {\bf Claim 2}. $(Z, \rho)$ is not weakly Whyburn. 

\begin{proof}[Proof of Claim 2] We claim that no almost closed set converges outside of $X'$. Indeed, let $A \subset X'$ be such that $\overline{A} \setminus X' \neq \emptyset$. Then $F^{-1}(A)$ is somewhere dense in $X$, and hence we can find a non-empty open set $U$ lying inside $\overline{F^{-1}(A)}$. But then $U \subset \overline{A} \setminus A$. Since $X$ does not have any isolated points, the set $U$ has infinite cardinality, and that implies that $A$ is not an almost closed set.
\renewcommand{\qedsymbol}{$\triangle$}
\end{proof}
\end{proof}

\begin{theorem}
There is a countably compact space of cardinality $\aleph_1$ which is not weakly Whyburn.
\end{theorem}

\begin{proof}
Let $Y=\omega_1$ and identify the set of all successor ordinals with the space $Z$ defined in Lemma $\ref{lemexample}$. We define a topology $\tau$ on $Y$ by declaring a set $U \subset Y$ to be open if $U$ is open in the order topology and $U \cap Z$ is open in $(Z, \rho)$.

\noindent {\bf Claim 1}. $(Y, \tau)$ is sequentially compact.

\begin{proof}[Proof of Claim 1]
 Let $S \subset Y$ be a countable infinite set. Pick $x_n \in S$ such that $x_{n-1}<x_n$, for every $n \geq 1$. Then $\{x_n: n < \omega \}$ converges to $\sup \{x_n: n < \omega \}$ in the order topology. But then this is true for the coarser topology of $(Y, \tau)$ as well.
\renewcommand{\qedsymbol}{$\triangle$}
\end{proof}

\noindent {\bf Claim 2}: $(Y, \tau)$ is a Hausdorff space.

\begin{proof}[Proof of Claim 2] Let $x, y \in Y$ be distinct points. Let $I$ and $J$ be disjoint intervals in the order topology such that $x \in I$ and $y \in J$. Then $I \cap Z$ and $J \cap Z$ are disjoint countable subsets of $Z$. Since countable sets are closed in $Z$ we can find disjoint sets $U$ and $V$,open in $Z$, such that $I \cap Z \subset U$ and $J \cap Z \subset V$. So $I \cup U$ and $J \cup V$ are disjoint open subsets of $Y$ which separate $x$ and $y$.
\renewcommand{\qedsymbol}{$\triangle$}
\end{proof}

\end{proof}

Also, the cardinality of the continuum is the smallest possible in Theorem $\ref{mainwhyburn}$, as the following well-known example shows (details are included for the reader's convenience).

\begin{example}
A countably compact regular space of cardinality continuum which is not weakly Whyburn.
\end{example}

\begin{proof}
Consider the \v Cech-Stone compactification of the integers $\beta \omega$ and set $A_0=\omega$.

Let $\beta < \omega_1$ and suppose we have constructed subsets $\{A_\alpha: \alpha < \beta\}$ of $\beta \omega$ such that $|A_\alpha| \leq 2^{\aleph_0}$, for every $\alpha < \beta$. For every $A \in [\bigcup_{\alpha < \beta} A]^\omega$, choose an accumulation point $p_A \in \beta \omega$ and set $A_\beta=\{p_A : A \in [\bigcup_{\alpha < \beta} A_\alpha]^\omega \}$. Then $\bigcup \{A_\alpha: \alpha < \omega_1\}$ is a countably compact space without convergent sequences. But every countably compact weakly Whyburn space contains a convergent sequence. It suffices to note that the closure of every countable almost closed set is a compact countable (and hence second-countable) space.
\end{proof}

There are consistent examples of non-weakly Whyburn spaces of cardinality continuum that are even compact. For example, Juh\'asz and Szentmikl\'ossy construct in \cite{JS} a compact non-pseudoradial space of cardinality continuum. Since every compact weakly Whyburn space is pseudoradial, their example cannot be weakly Whyburn. However, we don't know a ZFC example of a compact non-weakly Whyburn space of cardinality continuum.

\section{Which weakly Whyburn spaces are pseudoradial?}

The following three theorems summarize what has been known so far about the relationship between the weak Whyburn property and pseudoradiality.

\begin{theorem} (Bella, \cite{B})
Every compact weakly Whyburn space is pseudoradial.
\end{theorem}

\begin{theorem} \label{bcstheorem} \cite{BCS}
Every Lindel\"of weakly Whyburn $P$-space $X$ such that $\psi(X) < \aleph_\omega$ is pseudoradial.
\end{theorem}

\begin{theorem} (Alas, Madriz-Mendoza and Wilson, \cite{AMW})
Every weakly Whyburn k-space is pseudoradial.
\end{theorem}

Angelo Bella asked us in private communication whether compactness can be weakened to countable compactness and regularity in the first theorem. We show that the answer is negative. We then address the problem of weakening the Lindel\"of property to countable extent in the second theorem.

\begin{definition}
A sequence $\{x_\alpha: \alpha < \kappa\} \subset X$ is called free if for every ordinal $\beta < \kappa$, $\overline{\{x_\alpha: \alpha < \beta \}} \cap \overline{\{x_\alpha: \alpha \geq \beta \}}=\emptyset$. The \emph{freeness of $X$} ($F(X)$) is defined as the supremum of the cardinalities of the free sequences in $X$.
\end{definition}

The following lemma is well-known, but we include a proof for the reader's convenience.

\begin{lemma}
For every space $X$, we have $F(X) \leq L(X) \cdot t(X)$.
\end{lemma}

\begin{proof}
Let $\kappa = L(X) \cdot t(X)$ and suppose by contradiction that $X$ contains a free sequence $F=\{x_\alpha: \alpha < \kappa^+\}$ of cardinality $\kappa^+$. Since $L(X) \leq \kappa$, the set $F$ has a complete accumulation point $x$. Now, by $t(X) \leq \kappa$, there is a set $C \subset F$ such that $|C| \leq \kappa$ and $x \in \overline{C}$. Let $\gamma<\kappa^+$ be an ordinal such that $C \subset \{x_\alpha: \alpha < \gamma\}$. Since $F$ is a free sequence, we have $x \notin \overline{\{x_\alpha: \alpha \geq \gamma\}}$ and that contradicts the fact that $x$ is a complete accumulation point of $F$.
\end{proof}

\begin{lemma} \label{bellalemma}
(A. Bella, \cite{BFree}) Let $X$ be a pseudoradial regular space. Then $t(X) \leq F(X)$.
\end{lemma}

\begin{theorem} \label{angeloquest}
There is a countably compact regular weakly Whyburn non-pseudoradial space.
\end{theorem}

\begin{proof}
Let $X=\Sigma(2^{\omega_2})=\{x \in 2^{\omega_2}: |x^{-1}(1)| \leq \aleph_0 \}$ and let $p \in 2^{\omega_2}$ be the point defined by $p(\alpha)=1$, for every $\alpha < \omega_2$. We will prove that $Y=X \cup \{p\}$ with the topology inherited from $2^{\omega_2}$ is the required example.

 It is well known that the space $X$ is Fr\'echet-Urysohn (see \cite{E}), and that implies both that $X$ is weakly Whyburn and that $X$ has countable tightness.

\noindent {\bf Claim}. $L(X) = \aleph_1$.

\begin{proof}[Proof of Claim]
Let $\mathcal{U}$ be an open cover of $X$. Without loss of generality we can assume that for every $U \in \mathcal{U}$, there is a finite partial function $\sigma: \omega_2 \to 2$ such that $U=\{x \in 2^{\omega_2}: \sigma \subset x \}$. The domain of $\sigma$ will then be called the \emph{support of $U$} and we will write $supp(U)=dom(\sigma)$.

Let $\theta$ be a large enough regular cardinal and $M$ be an $\omega$-covering elementary submodel of $H(\theta)$ such that $X, \mathcal{U}, \omega_2 \in M$ and $|M|=\aleph_1$.

We claim that $\mathcal{U} \cap M$ covers $X$. Indeed, let $x \in X$ be any point and let $A \in M$ be a countable set such that $x^{-1}(1) \cap M \subset A$. 

Let $Z=\{y \in X: (\forall \alpha \in \omega_2 \setminus A)(y(\alpha)=0) \}$. Then $Z \in M$ and $Z$ is a compact subspace of $X$. So there is a finite subfamily $\mathcal{V} \in M$ of $\mathcal{U}$ such that $Z \subset \bigcup \mathcal{V}$. Since $\mathcal{V}$ is finite, we have $\mathcal{V} \subset M$. It then follows that $\mathcal{U} \cap M$ covers $Z$. 

Let $a$ be the point such that $a(\alpha)=x(\alpha)$ for all $\alpha \in M \cap \omega_2$ and $a(\alpha)=0$ for all $\alpha \in \omega_2 \setminus M$. The fact that $x^{-1}(1) \cap M \subset A$ implies that $a \in Z$ and hence there is $U \in \mathcal{U} \cap M$ such that $a \in U$. Note that $supp(U)$ is a finite element of $M$ and hence $supp(U) \subset M$. But since $x$ and $a$ coincide on $M$ we then have that $x \in U$ as well, as we wanted.

This proves $L(X) \leq \aleph_1$, but we can't have $L(X)=\aleph_0$ because $X$ is countably compact non-compact. Hence $L(X)=\aleph_1$.

\renewcommand{\qedsymbol}{$\triangle$}
\end{proof}

Suppose by contradiction that $Y$ is pseudoradial. Note that $t(Y)=\aleph_2$, hence by Lemma $\ref{bellalemma}$, $Y$ contains a free sequence $F$ of size $\omega_2$. So $F \setminus \{p\}$ is a free sequence of size $\omega_2$ in $X$. But $F(X) \leq L(X) t(X)=\aleph_1 \cdot \aleph_0=\aleph_1$, and that is a contradiction.
\end{proof}

Angelo Bella noted that Example $\ref{angeloquest}$ has the best possible character. As a matter of fact, every initially $\kappa$-compact weakly Whyburn regular space of character at most $\kappa^+$ is pseudoradial (recall that a space is called \emph{initially $\kappa$-compact} if every open cover of cardinality $\leq \kappa$ has a finite subcover). Before proceeding to the proof of this, we prove a simple lemma about convergent sequences in general spaces that is behind many results of this kind.

\begin{lemma} \label{lempseudo}
Let $X$ be any space and let $x \in X$ be a point such that $\psi(x,X)=\chi(x,X)$. Then $X \setminus \{x\}$ contains a transfinite sequence converging to $x$.
\end{lemma}

\begin{proof}
Suppose $\chi(x,X)=\kappa$ and let $\{U_\alpha: \alpha < \kappa \}$ enumerate a local base at $x$. Inductively choose points $x_\beta \in X$ such that, for every $\beta < \kappa$, $x_\beta \in \bigcap \{U_\alpha: \alpha \leq \beta \} \setminus (\{x_\alpha: \alpha < \beta \} \cup \{x\})$. This can be done because $\psi(x,X)=\kappa$. Note that $\{x_\alpha: \alpha < \kappa \}$ converges to $x$.
\end{proof}

\begin{proposition}
Let $X$ be an initially $\kappa$-compact weakly Whyburn regular space such that $\chi(X) \leq \kappa^+$. Then $X$ is pseudoradial.
\end{proposition}

\begin{proof}
Let $A \subset X$ be a non-closed set and let $B$ be a subset of $A$ such that $\overline{B} \setminus A=\{x\}$. If $\psi(x, \overline{B}) \leq \kappa$, then let $\{U_\alpha: \alpha < \lambda \}$ be a minimal sized family of open subsets of $\overline{B}$ such that $\bigcap_{\alpha < \lambda} \overline{U_\alpha}=\{x\}$. Inductively choose points $x_\alpha \in \bigcap \{\overline{U_\beta}: \beta \leq \alpha \} \setminus (\{x_\beta: \beta < \alpha \} \cup \{x\})$, for every $\alpha < \lambda$. Since $\overline{B} \setminus A=\{x\}$, we have $x_\alpha \in A$, for every $\alpha<\lambda$. Note now that $\{x_\alpha: \alpha < \lambda \}$ converges to $x$. Indeed, let $V$ be an open neighbourhood of $x$ in $\overline{B}$. We have $\{x\} = \bigcap \{\overline{U_\alpha}: \alpha < \lambda\} \subset V$. As $\lambda \leq \kappa$, by initial $\kappa$-compactness of $\overline{B}$, there is a finite subset $F$ of $\lambda$ such that $\bigcap \{\overline{U_\alpha}: \alpha \in F \} \subset V$. If we let $\gamma=\max(F)$, we see that $\{x_\alpha: \alpha \geq \gamma \} \subset V$, and hence $\{x_\alpha: \alpha < \lambda \}$ converges to $x$.

If $\psi(x, \overline{B}) =\kappa^+$, then $\psi(x, \overline{B})=\chi(x, \overline{B})$ and hence we can use Lemma $\ref{lempseudo}$.
\end{proof}

However, we don't know an answer to the following question when $\kappa>\aleph_0$.

\begin{question}
Is there, for every cardinal $\kappa$, an initially $\kappa$-compact weakly Whyburn space of character at most $\kappa^{++}$ which is not pseudoradial?
\end{question}

We are greatly indebted to Ofelia Alas for her considerable simplification of the proof of Theorem $\ref{charomega2}$.

\begin{lemma} \label{lemchar}
Let $X$ be a regular $P$-space of countable extent and let $x$ be a point such that $\psi(x, X)=\omega_1$. Then $\chi(x,X)=\omega_1$. 
\end{lemma}

\begin{proof}
Recall that every regular $P$-space is zero-dimensional. Let $\{U_\alpha: \alpha < \omega_1 \}$ be a decreasing family of clopen sets such that $\bigcap \{U_\alpha: \alpha < \omega_1 \}=\{x\}$. We claim $\{U_\alpha: \alpha < \omega_1 \}$ is actually a local base at $x$. Suppose that this is not the case and let $W$ be an open neighbourhood of $x$ such that $U_\alpha \nsubseteq W$, for every $\alpha < \omega_1$. Use this to choose, for every $\alpha< \omega_1$ a point $x_\alpha \in U_\alpha \setminus W$. Use countable extent to find an accumulation point $z$ of the set $\{x_\alpha: \alpha < \omega_1 \}$. Since every countable subset of $X$ is closed we have that $z$ is actually a complete accumulation point of $\{x_\alpha: \alpha < \omega_1 \}$. Then $z \in \overline{\{x_\alpha: \alpha \geq \beta \}}$ for all $\beta < \kappa$ and hence $z \in \bigcap \{U_\alpha: \alpha < \kappa \}=\{x\}$. It follows that $z=x$, but that is a contradiction, because $W$ is an open neighbourhood of $x$ disjoint from $\{x_\alpha: \alpha < \omega_1 \}$.

\end{proof}

\begin{theorem} \label{charomega2}
Let $X$ be a regular weakly Whyburn $P$-space of countable extent such that $\chi(X) \leq \omega_2$. Then $X$ is pseudoradial.
\end{theorem}

\begin{proof}
Let $A \subset X$ be a non-closed set and let $B \subset A$ be a set such that $\overline{B} \setminus A=\{x\}$, for some $x \in \overline{A} \setminus A$.

If $\psi(x, \overline{B})=\omega_1$ then, applying Lemma $\ref{lemchar}$ we see that $\chi(x, \overline{B})=\omega_1$. Let $\{U_\alpha: \alpha < \omega_1\}$ be a decreasing local base for $x$ in $\overline{B}$. For every $\alpha < \omega_1$, let $x_\alpha$ be a point in $U_\alpha \cap \overline{B} \setminus \{x\}$. Since $\overline{B} \setminus A=\{x\}$, we actually have $x_\alpha \in A$ and hence $\{x_\alpha: \alpha< \omega_1 \}$ is a sequence inside $A$ converging to $x$:

If $\psi(x, \overline{B})=\omega_2=\chi(x, \overline{B})$ then Lemma $\ref{lempseudo}$ guarantees the existence of a sequence $\{x_\alpha: \alpha < \omega_2 \} \subset \overline{B} \setminus \{x\}$ converging to $x$. Since $\overline{B} \setminus A=\{x\}$ we actually have that $\{x_\alpha: \alpha < \omega_2 \} \subset A$ and hence we are done.
\end{proof}

The following simple example shows that Theorem $\ref{charomega2}$ is not covered by Theorem $\ref{bcstheorem}$.

\begin{example}
A weakly Whyburn $P$-space $X$ of countable extent which is not Lindel\"of.
\end{example}

\begin{proof}
Let $X=\{\alpha < \omega_2: cf(\alpha)>\aleph_0 \}$ with the topology inherited from the order topology on $\omega_2$. Then $X$ is clearly a $P$-space. To see why $X$ has countable extent, note that for every $A \in [X]^{\omega_1}$ there is $B \subset A$ such that $otp(B)=\omega_1$ and hence $sup(B)$ belongs to $X$ and is an accumulation point of $A$. To see why $X$ is weakly Whyburn, let $A$ be a non-closed set. Let $x=\min(\overline{A} \setminus A)$. If we set $B:=A \cap [0, x]$, we see that $\overline{B} \setminus A=\{x\}$.  Finally, the open cover consisting of all initial segments shows that $X$ is not Lindel\"of.
\end{proof}

The following natural question is left open.

\begin{question} \label{questext}
Is there a regular weakly Whyburn $P$-space $X$ of countable extent such that $\psi(X) < \aleph_\omega$ and $X$ is not pseudoradial?
\end{question}

In view of Lemma $\ref{lempseudo}$, an example answering positively to Question $\ref{questext}$ must also be a positive answer to the following question.

\begin{question}
Is there a regular $P$-space $X$ of countable extent with a point $x \in X$ such that $\psi(x,X) < \chi(x,X)$ and $\psi(X) < \aleph_\omega$?
\end{question}

\section{Acknowledgements}

The author is grateful to FAPESP for financial support through postdoctoral grant 2013/14640-1, \emph{Discrete sets and cardinal invariants in set-theoretic topology}. The author is grateful to Ofelia Alas, Angelo Bella and L\'ucia Junqueira for helpful discussion that led to an improvement of the exposition of the paper.

\end{document}